\pdfoutput=1

\documentclass[12pt,a4paper]{amsart}

\usepackage[utf8]{inputenc}
\usepackage[T1]{fontenc}
\usepackage{lmodern}
\usepackage[kerning, spacing, tracking]{microtype}

\usepackage[style=numams, 
            sorting=nyt, 
            isbn=false,  
            backref=true, 
            backend=biber]{biblatex}
\usepackage{amsmath,amssymb, amsthm}
\usepackage{mathtools}
\usepackage{enumitem}

\usepackage{tikz-cd}

\usepackage[colorlinks,  linkcolor=blue, citecolor=blue, pdfa]{hyperref}

\usepackage[hcentering, hscale=0.70, vscale=0.74]{geometry}

\swapnumbers
\newtheorem{thm}{Theorem}[section]
\newtheorem{cor}[thm]{Corollary}
\newtheorem{lemma}[thm]{Lemma}
\newtheorem{prop}[thm]{Proposition}

\theoremstyle{definition}
\newtheorem{defi}[thm]{Definition}

\theoremstyle{remark}
\newtheorem{remark}[thm]{Remark}

\newlist{enumthm}{enumerate}{1}  
\setlist[enumthm]{label= \textup{(\roman*)}}

\renewcommand{\phi}{\varphi}
\renewcommand{\rho}{\varrho}

\newcommand{\eps}{\varepsilon}

\renewcommand{\leq}{\leqslant}

\newcommand{\nbd}{\nobreakdash-\hspace{0pt}}  
\newcommand{\defemph}[1]{\textbf{#1}} 

\newcommand{\nats}{\mathbb{N}}   

\newcommand{\sym}[1]{\mathrm{S}_{#1}}  

\newcommand{\iso}{\cong}    

\DeclareMathOperator*{\bigdcup}{\mathaccent\cdot{\mathop{\bigcup}}}

\DeclarePairedDelimiter{\abs}{\lvert}{\rvert} 

\DeclareMathOperator{\Z}{\mathbf{Z}}
\DeclareMathOperator{\C}{\mathbf{C}}         
\DeclareMathOperator{\Gal}{Gal}  

\DeclareMathOperator{\enmo}{End} 

\DeclareMathOperator{\trd}{trd}  

\DeclareMathOperator{\BrCliff}{BrCliff}
\DeclareMathOperator{\Cores}{Cores}
\DeclareMathOperator{\Res}{Res}
\DeclareMathOperator{\Ind}{Ind}

\DeclareMathOperator{\Br}{Br}  

\hypersetup{pdfauthor={Frieder Ladisch},
            pdftitle={Corestriction for algebras with group action},
            pdfkeywords={Brauer-Clifford group, 
                         equivariant Brauer group,
                         corestriction,
                         tensor induction,
                         G-algebras,
                         Azumaya algebras
                        } 
            }

\begin{document}
\title{Corestriction for algebras with group action}
\author{Frieder Ladisch}
\address{Universität Rostock\\
         Institut für Mathematik\\
         Ulmenstr.~69, Haus~3\\
         18057 Rostock\\
         Germany}
\email{frieder.ladisch@uni-rostock.de}
\thanks{The author was 
partially supported by the DFG (Project: SCHU 1503/6-1)}
\subjclass[2010]{Primary 16K50, Secondary 16W22} 
\keywords{
   Brauer-Clifford group,
   equivariant Brauer group,
   corestriction,
   tensor induction,
   $G$-algebras,
   Azumaya algebras
}

\begin{abstract}
  We define a corestriction map for equivariant Brauer groups
  in the sense of Fröhlich and Wall, which contain as a special case
  the Brauer-Clifford groups introduced by Turull.
  We show that this corestriction map has similar properties as the corestriction
  map in group cohomology (especially Galois cohomology).
  In particular, 
  composing corestriction and restriction associated to 
  a subgroup $H\leq G$ amounts to
  powering with the index $\lvert G:H \rvert$.  
\end{abstract}
\maketitle

\section{Introduction}
The Brauer-Clifford group 
was introduced by 
Turull~\cite{turull09, turull11} 
to find correspondences
between certain families of irreducible characters of finite groups,
such that  fields of
values and Schur indices of corresponding characters are 
equal~\cite[cf.][]{turull08c}. 
The character families are usually characters lying over
some fixed irreducible character of a normal subgroup.
This is the topic of \emph{Clifford theory}.
A classical result in this context is a theorem of 
Gallagher which states that an invariant irreducible character
of a normal subgroup of a finite group
  extends to the whole group if and only if
it extends to every Sylow subgroup~\cite[Theorem 21.4]{huppCT}.
The proof uses the corestriction map of group cohomology.
The aim of this paper is to provide a similar tool
in the theory of the Brauer-Clifford group.
This will be applied in forthcoming work
in the Clifford theory over small fields~\cite[cf.][]{ladisch14c}.

Let $R$ be a commutative ring on which a group $G$
acts by ring automorphisms. 
As Herman and Mitra~\cite{HermanMitra11} have pointed out, 
the Brauer-Clifford group $\BrCliff(G,R)$ is a 
special case of the equivariant Brauer group introduced earlier
by Fröhlich and Wall~\cite[cf.][]{frohlichwall00}.
This group consists of certain 
equivalence classes of certain $G$-algebras over $R$.
By a $G$\nbd algebra over $R$ we mean an algebra
$A$ over $R$ on which $G$ acts such that
$r^g\cdot 1_A = (r\cdot 1_A)^g$ for all $r\in R$ and $g\in G$.

Let  $H\leq G$ be  a subgroup.
Then (by restriction) any $G$\nbd algebra can be viewed
as an $H$\nbd algebra.
This defines a map
\[ \Res=\Res^G_H\colon \BrCliff(G,R)\to \BrCliff(H, R)
\]
called restriction.
In this paper, we define a map
\[ \Cores=\Cores_H^G\colon \BrCliff(H,R)\to \BrCliff(G,R),
\]
which has analogous properties to the 
transfer or corestriction map from group cohomology.
This map will only be defined when the index
$\abs{G:H}$ is finite, and will be called 
\emph{corestriction}
or \emph{transfer} map.
The main result is that if $[A]$ is an equivalence class
of algebras in $\BrCliff(G,R)$, then we have
\[\Cores_H^G(\Res_H^G[A])=[A]^{\abs{G:H}}.
\]

We will first show how to construct a
$G$\nbd algebra $A^{\otimes G}$ over $R$,
given an arbitrary $H$\nbd algebra $A$ over $R$.
To do this, we use the well known concept
of \emph{tensor induction}.
We will review this concept, and give a simple conceptual definition
of it that seems not as widely known as it should be.
We will then show that tensor induction respects
equivalence of $G$\nbd algebras,
and thus defines a group homomorphism of 
Brauer-Clifford groups.
Finally, in the last section we prove the main result
mentioned above.

\section{On tensor induction}
Let $G$ be a group and $R$ a commutative 
$G$\nbd ring. 
By definition, 
a \defemph{$G$\nbd ring} is a ring $R$ on which
the group $G$ acts by ring automorphisms.
We use exponential
notation $r\mapsto r^g$ to denote this action.

An $R$\nbd module $V$ is called an $R$\nbd module with
compatible $G$\nbd action,
if $G$ acts on the abelian group $V$ 
such that
$(vr)g= vgr^g$ for all
$v\in V$, $r\in R$ and $g\in G$.
We often use exponential notation, that is,
we write $v^g$ instead of $vg$, 
in particular, when $V$ happens to be an $R$\nbd algebra.

For the moment, we continue to use right multiplicative
notation. 
Every $R$\nbd module with $G$\nbd action can 
be viewed as a right $RG$\nbd module,
where $RG$ denotes the \defemph{crossed product} of $R$
with $G$ 
(also called the \defemph{skew group ring}
 of $G$ over $R$).
This is the set of formal sums
\[ \sum_{g\in G}  gr_g , \quad r_g\in R,
\]
with multiplication induced by
$(g_1r_1)(g_2 r_2)= g_1g_2 r_1^{g_2}r_2$,
extended linearly.
Conversely, any right module over $RG$
can be viewed as an $R$\nbd module with
$G$\nbd action.

Let $H\leq G$ and let $V$ be a right $RH$\nbd module.
Since $RH\subseteq RG$, we may form
the induced module 
$V^G=\Ind_H^G(V)=V\otimes_{RH}RG$.
Set $\Omega= \{Hg\mid g\in G\}$
and $V_{\omega}=V\otimes g\subseteq V^G$
for $\omega=Hg\in \Omega$.
Then
$V_{\omega}$ depends only on the coset 
$\omega=Hg$,
not on the specific representative $g$.
Every $V_{\omega}$ is an $R$\nbd submodule
of $V^G$
and
\[ V^G = \bigoplus_{\omega\in \Omega} V_{\omega},
   \quad\text{with}\quad
   V_\omega g= V_{\omega g}
   \;\text{for all $\omega\in \Omega$}.
\] 
The modules $V_{\omega}$ and $V_{\omega g}$
are isomorphic as abelian groups, 
but in general not as $R$\nbd modules.
Instead, the map
$\rho_{\omega, g}\colon V_{\omega}\to V_{\omega g}$
given by $v\rho_{\omega, g}=vg$
has the property 
$(vr)\rho_{\omega, g} = v\rho_{\omega,g} r^g$.
Following Riehm~\cite{Riehm70}, we call
an $R$\nbd module $W$ a
\defemph{$g$-conjugate} of the $R$\nbd module
$V$, if there is an 
isomorphism $\kappa\colon V\to W$ of abelian groups
such that $(vr)\kappa=v\kappa r^g$ for all $v\in V$.
Thus $V\otimes g$ is a $g$\nbd conjugate of $V$.

The 
\defemph{tensor induced} module is, 
by  definition,
the tensor product
\[ V^{\otimes G}:=
   \bigotimes_{\omega\in \Omega} V_{\omega}
   \quad\text{(tensor product over $R$)}.
\]
This is made into an $RG$\nbd module by defining
\[ \left( \bigotimes_{\omega\in \Omega} v_{\omega}
   \right)g
   = \bigotimes_{\omega g\in \Omega}v_{\omega}g
   = \bigotimes_{\omega\in \Omega} v_{\omega g^{-1}} g,
\]
where $v_{\omega g^{-1}}g\in V_{\omega}$.
Some readers may prefer to enumerate
\[\Omega= \{Hg_1,Hg_2,\dotsc, Hg_n\}, 
\]
and then write
\[ (v_1\otimes\dotsm \otimes v_n)g
    = v_{1g^{-1}}g\otimes \dotsm \otimes v_{ng^{-1}}g
    \quad(v_i\in V_{Hg_i}).
\]
It is easy to see that this extends to a well defined action
of $G$ on the tensor induced module $V^{\otimes G}$, 
and that this action makes 
$V^{\otimes G}$ a module over
the skew group ring $RG$.

A word on notation: The isomorphism class of
$V^{\otimes G}$ depends, of course, on the subgroup $H$
and the ring $R$.
If we want to emphasize the dependence on the 
subgroup $H$, we write $(V_H)^{\otimes G}$.
If we want to emphasize the dependence on the ring
$R$, we write $V^{\otimes_R G}$.
Both notations can be combined.

We mention in passing that the definition of tensor induction
given here is independent of the choice of a set of 
coset representatives. 
Such a definition was asked for by Kovács~\cite{kovacs90}
and Pacifici~\cite{pacifici05}.
Our definition (for trivial action of $G$ on $R$)
appears in a paper of 
Knörr~\cite[Def. 10, attributed to the referee]{knoerr07}.
The usual definition~\cite[\S~13]{CRMRT1} 
yields an isomorphic module, as is not difficult to see.

The tensor induced module can be characterized by a universal property:
\begin{thm}\label{t:tensorinduniv}
  Let $H\leq G$ be groups and
  $V$ an $RH$\nbd module.
  Let $\tau\colon V^{G}\to V^{\otimes G}$
  be the canonical multilinear map.
  Let $W$ be an $RG$\nbd module and
  $\beta\colon V^{ G }\to W$ a map with the following
  properties:
  \begin{enumthm}
  \item $\beta$ is a map of $G$\nbd sets.
  \item $\beta$ is $R$\nbd multilinear, when
        $V^{G}$ is considered as direct product
        of the modules $V_{\omega}$
        ($\omega\in \Omega=\{Hg\mid g\in G\}$).
  \end{enumthm}
  Then there is a unique
    $RG$\nbd module homomorphism
    $\phi\colon V^{\otimes G}\to W$ making the diagram
    \[ \begin{tikzcd}
         V^{G} \arrow{r}{\tau} 
             \arrow{rd}{\beta}
          & V^{\otimes G} \arrow{d}{\phi}
          \\
          & W
       \end{tikzcd}
    \]
    commutative.
    The pair $(\tau, V^{\otimes G})$ is determined
    uniquely up to unique isomorphism by the fact that this
    holds for all maps $\beta\colon V^{G }\to W$
    as above.
\end{thm}
\begin{proof}
  The uniqueness of $(\tau, V^{\otimes G})$ will follow
  from the usual abstract nonsense argument, once we have 
  proved existence and uniqueness of $\phi$.
  
  By the universal property of the tensor product
  $\bigotimes_{\omega\in \Omega} V_{\omega}$,
  there exists a unique $R$\nbd module homomorphism
  $\phi\colon V^{\otimes G}\to W$ making the diagram commutative,
  and we need to show that $\phi$ commutes with the action
  of $G$.
  This is true on the pure tensors in $V^{\otimes G}$,
  since it is true for $\tau$ (by definition)
  and $\beta$ (by assumption).
  From this, the assertion follows.  
\end{proof}

The next result shows that $(\phantom{I})^{\otimes G}$
is a functor from the category of $RH$\nbd modules to
the category of $RG$\nbd modules.
\begin{prop}\label{p:ti_func}
  Let $\phi\colon V\to W$ be a homomorphism of
  $RH$\nbd modules.
  Then there is a unique $RG$\nbd module homomorphism
  $ \phi^{\otimes G}\colon V^{\otimes G}\to W^{\otimes G}$,
  such that the diagram
  \[ \begin{tikzcd}
        V^G \rar{\phi\otimes 1}
            \dar 
       &W^G \dar \\
       V^{\otimes G} \rar{\phi^{\otimes G}} 
       & W^{\otimes G}
     \end{tikzcd}
  \]
  is commutative. (Here $V^G=V\otimes_{RH}RG$.)
  If 
  $\psi \colon W\to U$
  is another $RH$\nbd module homomorphism,
  then 
  $(\phi\psi)^{\otimes G}=\phi^{\otimes G}\psi^{\otimes G}$.   
\end{prop}
\begin{proof}
  The first assertion follows from 
  Theorem~\ref{t:tensorinduniv}, 
  applied to the $\beta$ in
  \[ \begin{tikzcd} 
       V^G \rar{\phi\otimes 1}
           \arrow[bend right]{rr}{\beta}  
           & W^G  \rar 
           & W^{\otimes G}.
     \end{tikzcd}
  \]
  Of course, one can define $\phi^{\otimes G}$
  directly: Let $G= \bigdcup_{t\in T}Ht$ and
  $v_t\in V$.
  Then define
  \[ \left( \bigotimes_{t\in T} (v_t\otimes t)
     \right) \phi^{\otimes G}
     = \bigotimes_{t\in T}(v_t\phi \otimes t)
     \in \bigotimes_{t\in T} (W\otimes t)
     = W^{\otimes G}.
  \]
  The second assertion follows from the uniqueness
  of $\phi^{\otimes G}$, $\psi^{\otimes G}$
  and $(\phi\psi)^{\otimes G}$ in
  \[ \begin{tikzcd}
       V^G \rar{\phi\otimes 1}
           \dar 
      &W^G \rar{\psi\otimes 1}
           \dar
      &U^G \dar \\
       V^{\otimes G} \rar{\phi^{\otimes G}} 
           \arrow[bend right]{rr}[swap]{(\phi\psi)^{\otimes G}}
      &W^{\otimes G} \rar{\psi^{\otimes G}}
      &U^{\otimes G}\rlap{.}
    \end{tikzcd}
  \]
\end{proof}
\begin{remark}
  While the involved categories are 
  \emph{abelian} categories, the functor
  $(\phantom{I})^{\otimes G}$ is not 
  additive:
  In general
  $(\phi_1+\phi_2)^{\otimes G}
    \neq \phi_1^{\otimes G} + \phi_2^{\otimes G}$.
\end{remark}

\section{Corestriction of \texorpdfstring{$G$}{G}-algebras}
As before, let $R$ be a commutative $G$\nbd ring.
A \defemph{$G$\nbd algebra over $R$}
is a $G$\nbd ring $A$, which is at the same time
an $R$\nbd algebra and such that 
$A$ is an $R$\nbd module with compatible $G$\nbd action.
In other words, we have
$(ar)^g= a^gr^g$ for all $a\in A$, $r\in R$ and $g\in G$.
This can also be expressed by saying that the algebra unit
$R\to \Z(A)$ is a homomorphism of $G$\nbd rings.
Note that now we are using exponential notation
for the action of $G$ on $A$,
although we can view $A$ as $RG$\nbd module.

Let $H\leq G$.
Given an $H$-algebra $A$ over $R$, we may form
the tensor induced module
$A^{\otimes G}$.
Let $\Omega=\{Hg\mid g\in G\}$.
Then $A^{\otimes G}$ is a tensor product
of modules $A_{\omega}$, where
$A_{\omega}=A\otimes g\subseteq A\otimes_{RH}RG$ for $\omega=Hg$.
Since we prefer to use exponential notation in the algebra 
situation,
we write $A^{\otimes g}$ instead of 
$A\otimes g$ and
$a^{\otimes g}$ instead of $a\otimes g$ from now on.
Each $A_{\omega}=A^{\otimes g}$ is in fact an algebra in its own right,
via $(a^{\otimes g})(b^{\otimes g})=(ab)^{\otimes g}$.

For each $g\in G$ and $\omega = Ht\in \Omega$, we have
a ring isomorphism $A_{\omega}\to A_{\omega g}$
sending $a^{\otimes t}$ to 
$(a^{\otimes t})^g =a^{\otimes t g}$.
This is not an $R$\nbd algebra homomorphism, but we have
$(ra_{\omega})^g = r^ga_{\omega}^g$.
As $R$\nbd algebra, 
$A^{\otimes g} $ is a 
$g$\nbd conjugate of $A$.

Since every $A_{\omega}$ is an $R$\nbd algebra, 
the tensor induced module
$A^{\otimes G}$ is an $R$\nbd algebra.
The action of $G$ on $A^{\otimes G}$ makes
$A^{\otimes G}$ a $G$\nbd algebra over $R$.
\begin{defi}\label{d:cores}
  Let $A$ be an $H$-algebra over $R$.
  The \defemph{corestriction} of $A$ in $G$ is
  the $G$\nbd algebra $A^{\otimes G}$.
  We also write $\Cores(A)$ or 
  $\Cores_H^G(A)$ to denote this algebra.
\end{defi}
\begin{remark}
  This definition must not be confused with the algebra-theoretic
  definition of the corestriction map from the Brauer group
  of a field to the Brauer group of a 
  subfield~\cite{Riehm70,draxlskewfields,Tignol87}.
  Of course, the two definitions are closely related.
  Let $L/F$ be a finite separable field extension and
  let $E\supseteq L$ be a field such that
  $E/F$ is Galois.
  Let $G=\Gal(E/F)$ and $H=\Gal(E/L)$.
  Then $\abs{G:H}=\abs{L:F}$ is finite.
  Let $A$ be an algebra over the field $L$.
  Then $A\otimes_L E$ becomes an $H$\nbd algebra
  over the $H$\nbd field $E$ by defining
  $(a\otimes \lambda)^h = a\otimes \lambda^h$.
  Let
  $B= (A\otimes_L E)^{\otimes G}$ in the sense of
  Definition~\ref{d:cores}.
  To get the corestriction map that is used in the theory
  of the Brauer group,
  one has to take $\C_B(G)$, the centralizer of $G$ in $B$.
  This is an algebra over the field $F$.
  Its isomorphism class does not depend on the choice of the
  Galois extension $E$.
  In particular, if $A$ is central simple over $L$,
  then $\C_B(G)$ is central simple over $F$.
  This defines a map 
  $\Br(L)\to \Br(F)$, also called corestriction, and denoted
  by $\operatorname{cor}_{L/F}$.
  However, in this paper we will only deal with
  the corestriction of Definition~\ref{d:cores}.
\end{remark}
Let $\Omega=\{Hg\mid g\in G\}$, 
and let $\mu_{\omega}\colon A_{\omega}\to A^{\otimes G}$ be 
the canonical
algebra homomorphism into $A^{\otimes G}$, so
\[ a_{\omega}\mu_{\omega} = 
            1\otimes \dotsm  \otimes 1
              \otimes a_{\omega} \otimes
              1 \otimes \dotsm \otimes 1,\]
where $a_{\omega}$ occurs at position $\omega$, of course.
The action of $G$ on $A^{\otimes G}$
 is uniquely determined by the 
property
$(a_{\omega}\mu_{\omega})^g= a_{\omega}^{g}\mu_{\omega g}$
for all $\omega\in \Omega$, $a_{\omega}\in A_{\omega}$, 
and $g\in G$.
This, in turn, follows from 
$(a\mu_H)^g = a^{\otimes g}\mu_{Hg}$ for all $g\in G$
and $a\in A$
and the $G$-action property:
To see this, note that for 
$\omega= Ht$ and $a_{\omega}=a^{\otimes t}$, we
have
\[ (a_{\omega}\mu_{\omega})^g
     = (a^{\otimes t}\mu_{Ht})^g
     = \big((a\mu_H)^t\big)^g
     = (a\mu_H)^{tg}
     = a^{\otimes tg}\mu_{Htg}
     =a_{\omega}^{g}\mu_{\omega g}.
\]
We have thus proved the following theorem
\cite[cf.][Theorem~4]{Riehm70}:
\begin{thm}\label{t:cores}
  Assume the notation introduced above.
  Then  the action of $G$ on $A^{\otimes G}$
  is uniquely determined by the fact that the diagram
  \[ \begin{tikzcd}
       A \arrow{r}{()^{\otimes g}} \arrow{d}{\mu_H} 
        & A_{Hg} \arrow{d}{\mu_{Hg}}
        \\
       A^{\otimes G} \arrow{r}{()^g} 
       &A^{\otimes G}
      \end{tikzcd}
  \]
  is commutative for all $g\in G$.
\end{thm}
Moreover, we have~\cite[Theorem~4]{Riehm70}:
\begin{thm}\label{t:universal}
  The homomorphism
  $\mu=\mu_H\colon A\to A^{\otimes G}$ 
  is a homomorphism of $H$\nbd algebras,
  and has the following universal property:
  Whenever $\psi\colon A\to B$
  is an $H$-algebra homomorphism of
  $A$ into a $G$\nbd algebra $B$ such that
  $A\psi$ and $(A\psi)^g$ commute for all
  $g\in G\setminus H$, then there is a unique
  $G$-algebra homomorphism
  $\alpha\colon A^{\otimes G}\to B$ making the diagram
  \[ \begin{tikzcd}
        A \arrow{r}{\mu} 
          \arrow{rd}[swap]{\psi} 
          & A^{\otimes G} \arrow{d}{\alpha}
          \\
          & B          
     \end{tikzcd}
  \]
  commutative.  
\end{thm}
\begin{proof}
  That $\mu$ is a homomorphism of $H$-algebras
  follows from the commutativity of the diagram
  in Theorem~\ref{t:cores} for $g\in H$.
  
  Assume $\alpha\colon A^{\otimes G}\to B$
  as in the theorem exists.
  Let $\omega=Ht\in \Omega$ and
  $a_{\omega}= a^{\otimes t}\in A_{\omega}$.
  Then $a_{\omega}\mu_{\omega} = (a\mu)^t$.
  Thus we necessarily
  have 
  $a_{\omega}\mu_{\omega}\alpha= ((a\mu)^t)\alpha = (a\mu\alpha)^t
   = (a\psi)^t$.   
  Let $G=\bigdcup_{t\in T}Ht$ and $a_t\in A$.
  Then $\alpha$ as in the theorem must send
  \[ \bigotimes_{t\in T} a_t^{\otimes t}
     = \prod_{t\in T}(a_t\mu)^t
  \qquad\text{to} \qquad
   \prod_{t\in T} (a_t\psi)^t.\]
  Now check that this indeed defines an algebra homomorphism.  
\end{proof}
Given a homomorphism
$\alpha\colon A\to B$ of $H$\nbd algebras over $R$,
the homomorphism
$\alpha^{\otimes G}\colon 
 A^{\otimes G}\to B^{\otimes G}$
defined in Proposition~\ref{p:ti_func}
is a $G$-algebra homomorphism.
We denote it also by 
$\Cores(\alpha)$.
In the algebra case, we can characterize 
$\alpha^{\otimes G}$ more elegantly by:
\begin{prop}\label{p:funccores}
  Let $\alpha\colon A\to B$ a homomorphism
  of $H$\nbd algebras over $R$.
  Then there is a unique homomorphism
  $\alpha^{\otimes G}\colon A^{\otimes G}\to B^{\otimes G}$
  of $G$\nbd algebras over $R$ making 
  \[\begin{tikzcd}
      A \dar{\mu}\rar{\alpha} 
       & B \dar{\mu}
       \\
      A^{\otimes G} \rar{\alpha^{\otimes G}}
      & B^{\otimes G}
  \end{tikzcd}
  \]
  commutative.
\end{prop}
\begin{proof}
  This follows immediately from Theorem~\ref{t:universal}.
\end{proof}

The next two results are of course also 
true for modules instead of algebras.

\begin{prop}\label{p:tiprod}
  Let $A$ and $B$ be $H$-algebras over $R$.
  Then 
  \[ A^{\otimes G} \otimes_{R} B^{\otimes G}
     \iso (A\otimes_R B)^{\otimes G}
  \]
  as $G$\nbd algebras over $R$.
  This is a natural equivalence of functors.
\end{prop}
\begin{proof}
  The isomorphism is given by the map
  well-defined by 
  \[ \left( \bigotimes_{\omega\in \Omega} a_{\omega}\right)
     \otimes
     \left( \bigotimes_{\omega\in \Omega} b_{\omega}\right)  
     \mapsto
     \bigotimes_{\omega\in \Omega}(a_{\omega}\otimes b_{\omega}).
     \qedhere
  \]
\end{proof}
(This means that the 
functor $(\cdot)^{\otimes G}= \Cores$ is actually a functor
of monoidal categories.)

Let $\phi\colon R\to S$ be a homomorphism of $G$\nbd rings.
If $A$ is an $H$\nbd algebra over $R$, then
$A\otimes_R S$ is an $H$\nbd algebra over $S$.
\begin{prop}
  Let $\phi\colon R \to S$ be a homomorphism of
  $G$\nbd rings. 
  Then
  \[ (A\otimes_R S)^{\otimes_S G}
      \iso 
      (A^{\otimes_R G})\otimes_R S
  \]
  naturally.
\end{prop}
In other words: scalar extension and corestriction commute.
\begin{proof}
  The map
  \begin{align*} 
    (A\otimes_R S)\otimes_{SH}SG
     &\to (A  \otimes_{RH} RG)\otimes_R S
     ,\quad 
     \\
     (a\otimes s)^{ \otimes g} 
     &\mapsto
     (a^{\otimes g}) \otimes s^g,
  \end{align*}
  is an isomorphism of $SG$\nbd modules and sends the
  algebra
  $(A\otimes_R S)^{\otimes g}$ to the algebra
  $A^{\otimes g} \otimes_R S$.
  Thus 
  \[ (A\otimes_R S)^{\otimes_S G}
      \iso \sideset{}{_S}\bigotimes_{t}  (A^{\otimes t} \otimes_R S)
      \iso \left( 
                 \sideset{}{_R}\bigotimes_{t}  A^{\otimes t} 
           \right) \otimes_R S 
  \]
  as $S$\nbd algebras.
  These isomorphisms respect the action of $G$, 
  as is easily checked.
\end{proof}

\section{Review of the Brauer-Clifford group}
We recall the definition of the 
equivariant Brauer group
by Fröhlich and Wall~\cite{frohlichwall00}.
As Herman and Mitra~\cite{HermanMitra11} have pointed out,
the Brauer-Clifford group as defined by Turull~\cite{turull09} 
is a special case. 

Let $G$ be a group.
A $G$-algebra $A$ over the commutative $G$\nbd ring $R$
is called an \defemph{Azumaya $G$\nbd algebra over $R$}
if it is Azumaya over $R$ as an $R$\nbd algebra.

If $A$ and $B$ are Azumaya $G$\nbd algebras over the 
$G$\nbd ring $R$, then
$A \otimes_{R} B$ is an Azumaya $G$\nbd algebra over $R$.

The \defemph{Brauer-Clifford group} $\BrCliff(G,R)$
is the set of equivalence classes of 
Azumaya $G$\nbd algebras over $R$ under an equivalence relation
that will be defined below.
The multiplication is induced by 
the tensor product $\otimes_R$ of algebras.

Let $P$ be an $RG$\nbd module, where
$RG$ denotes, as before, the skew group ring of $G$ over $R$.
Then $\enmo_R P$ has a natural structure of a $G$\nbd algebra
over the $G$\nbd ring $R$.
The action of $G$ on $\enmo_R P$ is given by
$p \alpha^g = ((p g^{-1})\alpha )g$
for $p\in P $ and $\alpha\in \enmo_R P$.
If $P$ is a progenerator over $R$, then
$\enmo_R P$ is Azumaya over $R$.
A \defemph{trivial $G$\nbd algebra over $R$}
is an algebra of the form $\enmo_R P$, 
where $P$ is an $RG$\nbd module such that
$P_R$ is a progenerator.
Two $G$\nbd algebras $A$ and $B$ over $R$ are called
\defemph{equivalent} if there are
trivial $G$\nbd algebras $S$ and $T$ such that
$A\otimes_R S \iso A\otimes_R T$ as $G$\nbd algebras over $R$.

This definition 
yields the same equivalence classes as another definition
using equivariant Morita equivalence~\cite[Proposition~9]{HermanMitra11}. 

As mentioned before, the set of equivalence  classes
of Azumaya $G$\nbd algebras over $R$ forms 
the \defemph{equivariant Brauer group}, or, as we call it here,
the \defemph{Brauer-Clifford group}
$\BrCliff(G,R)$. 

\section{Corestriction of equivalent \texorpdfstring{$G$}{G}-algebras}
We continue to assume that $G$ is a group,
$H\leq G$ is a subgroup of finite index and
$R$ is a $G$\nbd ring.
\begin{lemma}\label{l:coresmodule}
  Let $P$ be a right $RH$\nbd module which is 
  finitely generated projective over $R$.
  Then 
  \[ \enmo_R(P^{\otimes G}) \iso (\enmo_R P)^{\otimes G}
  \]
  as $G$\nbd algebras over $R$.
\end{lemma}
\begin{proof}
  Write $S= \enmo_R P$.
  Observe that $S$ is an $H$\nbd algebra over $R$.
  We have 
  \[P^{\otimes G}= \bigotimes_{\omega\in \Omega}P_{\omega}
    \quad\text{and}\quad
    S^{\otimes G} = \bigotimes_{\omega\in\Omega}S_{\omega}.
  \]
  Let $t\in \omega$ and $s\in S$.
  We define an action of $S_{\omega}$ on $P_{\omega}$
  by
  \[ (x \otimes t)s^{\otimes t} = (xs)\otimes t.
  \]
  This is independent of the choice of $t\in \omega$
  since 
  \begin{align*} 
   (x \otimes (ht) )s^{\otimes ht}
    = (xs)\otimes (ht)
    &= ((xs)h) \otimes t
    \\
    &= (xhs^h)\otimes t
     = \big((xh)\otimes t \big) (s^h)^{\otimes t}
     .
  \end{align*}
  We can identify $S_{\omega}$ with $ \enmo_R P_{\omega}$
  via this action.
    Since $P_R$ is finitely generated projective,
    we have 
    \[\enmo_R(P^{\otimes G}) 
      = \enmo_R \left( \bigotimes_{\omega\in \Omega} 
                       P_{\omega}
                \right)
      \iso \bigotimes_{\omega\in \Omega}
            \enmo_R P_{\omega}
      \iso \bigotimes_{\omega\in \Omega} S_{\omega}
      = S^{\otimes G}
    \]
    as $R$\nbd algebras,
  where $S^{\otimes G}$ acts on
  $P^{\otimes G}$ by
  \[\left(\bigotimes_{\omega\in \Omega}x_{\omega}\right)
    \left(\bigotimes_{\omega\in \Omega}s_{\omega}\right)
    = \bigotimes_{\omega\in \Omega}(x_{\omega}s_{\omega}).
  \]
  The isomorphism above commutes with the action of $G$,
  as is easy to check. 
  This finishes the proof.
\end{proof}
\begin{prop}\label{p:equivcores}
  Let $R$ be a commutative $G$-ring
  and $H\leq G$.
  Let $A$ and $B$ be equivalent $H$\nbd algebras
  over $R$.
  Then $A^{\otimes G}$ and $B^{\otimes G}$
  are equivalent as $G$\nbd algebras over $R$.
\end{prop}
\begin{proof}
  Let $P$ and $Q$ be $RH$\nbd modules
  that are $R$\nbd progenerators and such that
  $A\otimes_R \enmo_R P\iso B\otimes_R \enmo_R Q$
  as $H$\nbd algebras over $R$.
  Then
  \[ (A\otimes_R \enmo_R P)^{\otimes G}
     \iso
     (B\otimes_R \enmo_R Q)^{\otimes G}
  \] 
  as $G$\nbd algebras over $R$ by Proposition~\ref{p:funccores}.
  By Lemma~\ref{l:coresmodule} and Proposition~\ref{p:tiprod},
  \[ A^{\otimes G} \otimes_R \enmo_R(P^{\otimes G})
     \iso A^{\otimes G} \otimes_R (\enmo_R P)^{\otimes G}
     \iso (A\otimes_R \enmo_R P)^{\otimes G}
  \]
  and similarly for $B$ and $Q$.
  Thus
  \[ A^{\otimes G} \otimes_R \enmo_R(P^{\otimes G})
     \iso 
     B^{\otimes G} \otimes_R \enmo_R (Q^{\otimes G}),
  \]
  which shows that
  $A^{\otimes G}$ and $B^{\otimes G}$ are equivalent.   
\end{proof}
\begin{cor}
  Corestriction defines a group homomorphism
  \[\Cores=\Cores_H^G\colon \BrCliff(H,R)\to \BrCliff(G,R).
  \]
\end{cor}
\begin{proof}
  As $R$\nbd algebra, 
  $\Cores(A)$ is a $\abs{G:H}$\nbd fold tensor product
  of $A$ over $R$.
  Thus if $A$ is Azumaya over $R$, then $\Cores(A)$ is, too.
  By Proposition~\ref{p:equivcores}, 
  $\Cores$ as mapping of Brauer-Clifford groups
  is well defined.
  By Proposition~\ref{p:tiprod}, it is
  a group homomorphism.
\end{proof}
\section{Corestriction and restriction}
In this section we prove that
$\Cores_H^G(\Res_H^G[A]) = [A]^{\abs{G:H}}$
when $A$ is an Azumaya $G$\nbd algebra over $R$. 
The proof follows Tignol's proof of the same result
for the Brauer group over a field~\cite{Tignol87}.

Let $A$ be an Azumaya algebra over the commutative ring
$R$. 
We need a property of the
\emph{reduced trace}
$\trd=\trd_{A/R}\colon A\to R$~\cite[IV.2]{KnusOj74}
which is probably well known.
\begin{lemma}\label{l:invar_trd}
  Let $A$ be an Azumaya algebra over the commutative ring $R$
  and $g$ a ring automorphism of $A$
  (which restricts to a ring automorphism of
  $R\iso \Z(A)$).
  Then 
  \[ \trd(a^g) = \trd(a)^g
      \quad \text{for all $a\in A$}.
  \]
\end{lemma}
\begin{proof}
  The reduced trace
  commutes with scalar extensions.  
  This means: Let
  $\phi\colon R\to S$ be a ring homomorphism
  (with $\phi(1_R)= 1_S$).
  The homomorphism $\phi$ makes $S$ into an $R$\nbd module.
  The algebra $A\otimes_R S$ is Azumaya over $S$ and
   we have
  \[ \trd_{A\otimes_R S/S}(a\otimes 1)= \trd_{A/R}(a)^{\phi}.
  \]
  (This follows from the proof that the characteristic polynomial 
   and the reduced trace are well-defined~\cite[Proposition~IV.2.1]{KnusOj74}.) 
  We will apply this fact to
  $\phi=g_{|R}\colon R\to R$.
  We write $A\otimes_{R,\phi}R$ for the corresponding 
  tensor product, to emphasize the dependence on $\phi$.

  Second, if $f\colon B\to A$ is an isomorphism
  of Azumaya $R$\nbd algebras, then
  \[\trd_{A/R}(f(b))=\trd_{B/R}(b)
  \] 
  \cite[Lemme~IV.2.2]{KnusOj74}.  
  We apply this to the map $f\colon B=A\otimes_{R,\phi}R\to A$
  defined by
  $f(a\otimes r) = a^g r$.
  Note that $f$ is well-defined since
  $ar\otimes s = a\otimes r^{\phi}s$
  in $A\otimes_{R,\phi}R$ and so
  $f(ar\otimes s)= (ar)^gs = a^g r^{\phi}s
   = f(a\otimes r^{\phi}s)$.
  It is easy to check that $f$ is an isomorphism of $R$\nbd algebras,
  the inverse is given by 
  $a\mapsto a^{g^{-1}}\otimes 1$.
  
  Now the composition
  \[ \begin{tikzcd}
        A \rar
         & A\otimes_{R,\phi} R \rar{f}
         & A
     \end{tikzcd}
  \]
  yields the map $a\mapsto a^g$.
  Applying the above statements, we get
  \begin{align*}
    \trd_{A/R}(a^g) = \trd_{A/R}(f(a\otimes 1))
                    = \trd_{B/R}(a\otimes 1)
                   &= \trd_{A/R}(a)^{\phi}
                   \\
                   &= \trd_{A/R}(a)^g
  \end{align*}
  as claimed.
\end{proof}
We need another lemma. We write $\sym{n}$ to denote
the symmetric group on $n$ letters.
\begin{lemma}\label{l:switchhom}
  Let $A$ be an Azumaya $G$\nbd algebra over 
  the $G$\nbd ring $R$ and $n\in \nats$.
  Then there is a group homomorphism
  \[\sigma\colon \sym{n}\to 
     ( \underbrace{A\otimes_R \dotsm \otimes_R A}_{n} )^*
  \]
  such that 
  \[ \sigma(\pi)^{-1}
      (a_1 \otimes \dotsm \otimes a_n)\sigma(\pi)
       = a_{1\pi^{-1}}\otimes \dotsm \otimes a_{n\pi^{-1}}
  \]
  and $\sigma(\pi)^g=\sigma(\pi)$ for all 
  $\pi\in \sym{n}$ and $g\in G$.
\end{lemma}
\begin{proof}
  Consider first the case where $n=2$.
  Since $\enmo_R(A)\iso A\otimes_R A^{op}$,
  there is a unique element 
  $t=\sum_{i}x_i\otimes y_i\in A\otimes_R A$ 
  such that
  $\trd(a)1_A= \sum_{i} x_i a y_i$ for 
  all $a\in A$. 
  By a result of Goldman, 
  this element has the properties
  $t^2 = 1$ and $(a\otimes b)t=t(b\otimes a)$
  for all $a$, $b\in A$~\cite[Proposition~IV.4.1]{KnusOj74}.
  To finish the case $n=2$, 
  it remains to show that $t^g= t$ for $g\in G$.
  By Lemma~\ref{l:invar_trd}, we have
  $\trd(a^g)=\trd(a)^g$ for $g\in G$.
  This means that
  \[ \sum_{i} x_i a^g y_i = \trd(a^g)=\trd(a)^g 
       = \sum_{i} x_i^g a^g y_i^g
  \]
  for all $a\in A$. 
  Thus 
  \[ t = \sum_i x_i\otimes y_i  
       = \sum_i x_i^g\otimes y_i^g = t^g,\]
  as desired.
  
  Now for the general case. 
  By what we have done already, for every pair
  $(i,j)$ there is an element 
  $t_{ij}\in (A^{\otimes n})$ such that the inner automorphism
  induced by $t_{ij}$ switches the positions $i$ and $j$,
  and such that $t_{ij}$ is centralized by $G$.
  We first define the homomorphism $\sigma$
  on the
  neighbor transpositions $(i,i+1)$,
  setting $\sigma((i,i+1))=t_{i,i+1}$.
  To show that this extends to a homomorphism of the 
  symmetric group $S_n$ into $(A^{\otimes n})^*$,
  we use the fact that $S_n$ is a Coxeter group
  generated by the neighbor transpositions with relations
  $\big((i,i+1)(k,k+1)\big)^{m(i,k)}=1$,
  where $m(i,k)$ is the order of the corresponding element. 
  Thus we have to check three types of relations.
  The first is
  $t_{i,i+1}^2= 1$, which follows from the result of
  Goldman cited before.
  The second type of relation to check is 
  $(t_{i,i+1}t_{k,k+1})^2=1$
  whenever $\{i,i+1\}\cap\{k,k+1\}=\emptyset$.
  This relation is clear in view of $t_{i,i+1}^2=1$
  and since $t_{i,i+1}$ and $t_{k,k+1}$ live in different 
  components of $A^{\otimes n}$.
  Finally, we have to check that
  $(t_{i,i+1}t_{i+1,i+2})^3=1$.
  To do this, one can proceed as Knus and Ojanguren
  in their proof of $t^2= 1$~\cite[Proposition~IV.4.1]{KnusOj74}
  and reduce the problem to the case where
  $A$ is a matrix ring over $R$.
  Then in terms of matrix units $e_{rs}$, we have
  $t=\sum_{r,s} e_{rs}\otimes e_{sr}$.
  It suffices to compute in $A\otimes A\otimes A$, where
  we have to check that
  \[ \left( \left(\sum_{r,s} e_{rs}\otimes e_{sr}\otimes 1
            \right)
            \left(\sum_{u,v} 1\otimes e_{uv}\otimes e_{vu}
            \right)
     \right)^3 = 1\otimes 1 \otimes 1.
  \]
  We leave this simple computation to the reader.
\end{proof}
The existence of a homomorphism
$\sigma\colon S_n\to (A^{\otimes n})^*$
was also proved by Haile~\cite[Lemma~1.1]{Haile79}
and Saltman~\cite[Theorem~2]{Saltman81},
but we need the additional property that
$\sigma(\pi)$ is centralized by the group $G$.
Note that since $G$ is completely arbitrary,
the lemma says that all ring automorphisms
of $A$ centralize the image of $\sigma$.
\begin{thm}\label{t:rescores}
  Let $A$ be an Azumaya $G$\nbd algebra over 
  the $G$\nbd ring $R$ and $H\leq G$.
  Then $(A_H)^{\otimes G}$ and
  $A^{\otimes\abs{G:H}}$ are equivalent as
  $G$\nbd algebras over $R$.
  In other words, 
  \[ \begin{tikzcd}
       \BrCliff(G,R)\rar{\Res_H^G} 
       &\BrCliff(H,R)\rar{\Cores_H^G}
       &\BrCliff(G,R)
     \end{tikzcd}
  \]
  sends $[A]$ to $[A]^{\abs{G:H}}$.
\end{thm}
\begin{proof}
  Write $C=(A_H)^{\otimes G}$ and $B= A^{\otimes \abs{G:H}}$.
  We have to show that $B$ and $C$ are equivalent $G$\nbd algebras.
  First we show that $C$ and $B$ are isomorphic as $R$\nbd algebras.
  As in the construction of the tensor induced algebra, 
  let $\Omega= \{Hg\mid g\in G\}$ and
  $A_{\omega} = A^{\otimes g}\subseteq A\otimes_{RH}RG$,
  when $\omega=Hg$.
  Every element of $A_{\omega}$ has the form
  $a^{\otimes g}$ with $a\in A$.
  Note that $a^{\otimes g}\mapsto a^g$ yields an isomorphism
  $\phi_{\omega}\colon A_{\omega}\to A$ of $R$\nbd algebras
  which is independent of the choice of $g\in \omega$.
  We can view $B$ as a tensor product of copies of $A$ indexed
  by $\Omega$.
  Define
  $\phi\colon C\to B$
  by
  \[ \phi\left( \bigotimes_{\omega} a_{\omega} \right) 
    = \bigotimes_{\omega} \phi_{\omega}a_{\omega}.
  \]
  Then $\phi$ is a well-defined $R$\nbd algebra isomorphism.
  The problem is that $\phi$ does not commute with 
  the $G$\nbd actions on $C$ and $B$, respectively.
  
  Let 
  $\sigma$ be the group
  homomorphism from $\sym{{\Omega}}$, 
  the group of permutations of the set $\Omega$,
  into the centralizer of $G$ in $B^*$
  from Lemma~\ref{l:switchhom}.
  The action of $G$ on $\Omega$
  yields a group homomorphism from $G$ into $\sym{{\Omega}}$.
  Let $\pi\colon G\to \sym{{\Omega}}\to B^*$
  be the composition.
  Thus
  \[ \pi(g)^{-1}
     \left(\bigotimes_{\omega}b_{\omega}\right)
     \pi(g)
     = \bigotimes_{\omega} b_{\omega g^{-1}},\]
  where $b_{\omega g^{-1}}\in A$ occurs at position $\omega$
  in the last tensor.
  Thus
  \begin{align*}
   \phi\left( \bigotimes_{\omega}a_{\omega}\right)^{\pi(g)g}
       &= \left( \bigotimes_{\omega} \phi_{\omega}a_{\omega}\right)^{\pi(g)g}
        = \bigotimes_{\omega} \big(\phi_{\omega g^{-1}}(a_{\omega g^{-1}})\big)^g
       \\
       &= \bigotimes_{\omega} \phi_{\omega}\big((a_{\omega g^{-1}})^g\big)
       = \phi\left( \bigotimes_{\omega} (a_{\omega g^{-1}})^g \right)
       \\
      &= \phi\left( \big(\bigotimes_{\omega}a_{\omega}\big)^g \right).  
  \end{align*}
  Thus $\phi(c^g) = \phi(c)^{\pi(g)g}$ for all $c\in C$.
  
  For $b\in B$ and $g\in G$, set
  $b\star g = b^g\pi(g)$.
  This defines a new action of $G$ on $B$:
  \[ \left(b{\star g}\right){\star h}
    = (b^g\pi(g))^h\pi(h)= b^{gh}\pi(g)^h\pi(h)
    = b^{gh}\pi(gh)= b{\star(gh)},\]
  where we have used that $\pi$ is a group homomorphism
  and that $h\in G$ centralizes $\pi(g)$.
  Furthermore, for $r\in R$ we have $(br)\star g = (b\star g)r^g$.
  Thus $B$ is a right $RG$\nbd module via the 
  star action.
  
  Since $B$ is Azumaya over $R$, the natural homomorphism
  $B^{\text{op}}\otimes_R B\to\enmo_R B$ 
  is an isomorphism.
  Thus the map
  \[ \eps\colon B^{\text{op}}\otimes_R C \to \enmo_R B,
     \quad
     x (b\otimes c)^{\eps} = bx \phi(c)
     \quad(x\in B, \, b\in B^{\text{op}}, \, c\in C),
  \]
  is an isomorphism of $R$\nbd algebras.
  We claim that $\eps$ is an isomorphism
  of $G$\nbd algebras, where
  the $G$\nbd algebra structure of $\enmo_R B$ is that
  induced by the $RG$\nbd module structure of $B$
  defined by the star action.
  Namely, we have
  \[ x (b^g\otimes c^g)^{\eps} = b^gx \phi(c^g)
       = b^g x \phi(c)^{\pi(g)g}
  \]
  and 
  \begin{align*}
    x \left( (b\otimes c)^{\eps}\right)^g 
     = \big( (x\star g^{-1}) (b\otimes c)^{\eps}\big)\star g
     &= \big( bx^{g^{-1}}\pi(g^{-1})\phi(c)\big)^g\pi(g)
     \\
     &= b^g x \phi(c)^{\pi(g)g}.
  \end{align*}
  The claim follows.
  Thus $B^{\text{op}}\otimes_R C$
  is isomorphic to the trivial $G$\nbd algebra
  $\enmo_R B$.
  It follows that $B$ and $C$ are equivalent,
  as was to be shown.
\end{proof}
\begin{cor}
  Let $R$ be a commutative $G$-ring, $G$ a finite group.
  Let $p$ be a prime and $P$ a Sylow $p$-subgroup of $G$.
  Then the $p$\nbd torsion part
  $\BrCliff(G,R)_p$ of the Brauer-Clifford group
  is isomorphic to a subgroup of
  $\BrCliff(P, R)$.
\end{cor}
\begin{proof}
  The map $\Res^G_P\colon \BrCliff(G,R)\to \BrCliff(P,R)$
  is injective when restricted to 
  the $p$-torsion part,
  since $\Cores\circ \Res$ is, by Theorem~\ref{t:rescores}.
\end{proof}
\begin{cor}
  Let
  $A$ be an Azumaya $G$\nbd algebra over the 
  commutative $G$\nbd ring $R$, where $G$ is finite.
  Write $[A]$ for the equivalence class of $A$ in
  $\BrCliff(G,R)$.
  The following are equivalent:
  \begin{enumthm}
  \item $[A]= 1$
  \item $[\Res^G_P(A)]= 1$ in
          $\BrCliff(P,R)$
        for all Sylow subgroups $P$ of $G$.
  \item For every prime $p$ there is a Sylow
        $p$\nbd subgroup $P$ such that
        $[\Res^G_P(A)]= 1$ in
        $\BrCliff(P,R)$.
  \end{enumthm}
\end{cor}
The next corollary is of course known, 
it also follows from the fact that the kernel
of $\BrCliff(G,R)\to\Br(R)$ is isomorphic
to a cohomology group $H^2(G, \mathcal{C}_R)$ 
for a certain abelian group 
$\mathcal{C}_R$~\cite[Theorem~4.1]{frohlichwall00}.
\begin{cor}
  Let $G$ be a finite group and $R$ a $G$\nbd ring. 
  Then the Brauer-Clifford group $\BrCliff(G,R)$ is torsion.  
\end{cor}
\begin{proof}
  We have a natural homomorphism
  $\Res\colon \BrCliff(G,R)\to \Br(R)$.
  The Brauer group $\Br(R)$ is torsion, 
  as is well known~\cite{KnusOj74,Saltman81}.
  So if $[A]^n = 1$ in $\Br(R)$, then
  $[A]^{n\abs{G}}=1$ in $\BrCliff(G,R)$.
\end{proof}

\section*{Acknowledgment}
  I thank Jason Starr for his help with the proof of
  Lemma~\ref{l:invar_trd} 
  via 
  \href{http://mathoverflow.net/q/139250}{mathoverflow}~\cite{Starr13_MO139250}.

%
\printbibliography   
%
\end{document}